\documentclass[11pt,twoside]{article}
\usepackage{fullpage}
\usepackage{amssymb,amsfonts,amsmath,amsthm}
\usepackage{ stmaryrd}
	\theoremstyle{definition}
	\newtheorem{thm}{Theorem}
	\newtheorem{prop}[thm]{Proposition} 
	
	\newtheorem{cor}[thm]{Corollary}
	\newtheorem{rmk}[thm]{Remark}

\usepackage{enumerate}
\usepackage{hyperref}
\numberwithin{equation}{section}
\allowdisplaybreaks
\newcommand{\ignore}[1]{}
\newcommand{\idf}{{\tt IdentityFinder}}

\begin{document}
\thispagestyle{empty}
\title{{\idf} and some new identities of Rogers-Ramanujan type}
\author{Shashank Kanade and Matthew C. Russell
\footnote{Department of Mathematics, Rutgers, The State University of New Jersey, Piscataway, NJ
08854. Emails: {\tt \{skanade, russell2\} [at] math [dot] rutgers [dot] edu}}}
\date{}
\maketitle

\begin{abstract}
The Rogers-Ramanujan identities and various analogous identities
(Gordon, Andrews-Bressoud, Capparelli, etc.) form a family of very 
deep identities concerned with integer partitions. 
These identities (written in generating function form) are typically 
of the form ``product side'' equals ``sum side,'' with the 
product side enumerating partitions obeying certain congruence 
conditions and the sum side obeying certain initial conditions 
and difference conditions (along with possibly other restrictions). 
We use symbolic computation to generate various such sum sides 
and then use Euler's algorithm to see which of them actually 
do produce elegant conjectured product sides. 
We not only rediscover many of the known identities 
but also discover some apparently new ones, as conjectures
supported by strong mathematical evidence.
\end{abstract}

\section{Introduction}

The aim of this paper is to serve as an introduction to the Maple package
\idf, a package written to assist in the discovery of new partition identities.
It is freely available for download at:

\url{http://math.rutgers.edu/~russell2/papers/partitions14.html}.

Let us recall the celebrated pair of Rogers-Ramanujan identities
(cf.\ Chapter 7 of \cite{A1}). Let $n$ be a positive integer.

\begin{enumerate}
\item The number of partitions of $n$ into
parts congruent to $\pm 1$ modulo 5 is the same
as the number of partitions of $n$ such that
adjacent parts have difference at least 2.

\item The number of partitions of $n$ into
parts congruent to $\pm 2$ modulo 5 is the same
as the number of partitions of $n$ such that
adjacent parts have difference at least 2
and such that 1 does not appear as a part.
\end{enumerate}

These identities have a fascinating history (see for instance, \cite{A1}) 
and have stimulated major research programs in the past century.
In generating function form, the Rogers-Ramanujan identities
can be written as follows:

\begin{align}
\prod_{j\ge 0}
\frac 1 {\left(1-q^{5j+1}\right)\left(1-q^{5j+4}\right)} 
&= \sum_{n\ge 0}d_1(n)q^n,\\
\prod_{j\ge 0}
\frac 1 {\left(1-q^{5j+2}\right)\left(1-q^{5j+3}\right)} 
&= \sum_{n\ge 0}d_2(n)q^n,
\end{align}
where $d_i(n)$ is the number of partitions of $n$ such that
adjacent parts have difference at least 2 and such that the smallest
allowed part is $i$.

In this paper we focus on identities that have a similar shape, namely,
the ones in which a sum side related to restricted partitions clicks
into being an interesting infinite product.

Typically, the products in such identities correspond to congruence condititions
(analogues of the mod 5 conditions in the Rogers-Ramanujan identities)
and the sum sides are variations on the difference-at-a-distance
theme. 
Some notable identities that are analogous to 
the Rogers-Ramanujan identities and have this shape are: 
Euler's identities, Gordon's identities, the
Andrews-Bressoud identities, Capparelli's identities,
the little G\"{o}llnitz identities, etc.
Recently, in \cite{N}, partition identities involving
highly complicated, yet very beautiful, sum sides have been conjectured.

The aim of this paper is to
build a systematic mechanism based on symbolic computation 
that aids in the discovery of more such partition identities.
This mechanism is embodied in our Maple package \idf.
We present six new conjectured identities which were found using  \idf.

Let us elaborate on the inner workings of \idf. Our main idea is to 
generate a zoo of sum sides and then use Euler's algorithm to 
``factorize'' the sum sides into infinite products.  We then glean 
the results to find out which sum sides give ``periodic'' factorizations. 
These pairs of sum sides and their corresponding periodic factorizations 
are the worthwhile partition identities. As mentioned before, the sum sides
correspond to a variety of restrictions on the partitions. To name a few:
difference-at-a-distance condition, congruence-at-a-distance condition,
smallest part(s) condition, etc. We therefore write procedures corresponding to
each such condition, and then sift partitions which satisfy a combination of
such condititions with varying parameters.  This highly ``modular'' approach
makes it quite easy to enlarge the ``search space''  by incorporating a 
plethora of such checks, and advances in partition identities could be quickly
assimilated into our package. We are now working towards enlarging the 
search space to include multi-color partition identities
and to include highly complex sum sides such as those in \cite{N}. 

We are in a process of providing proofs of these identities.
There is no doubt in our minds that our new identities are true. 
Our original conjectures were based on calculations up to approximately $n=30$; however,
once potential cadidates were found, we were able to check the identities up to 
partitions of $n=500$. We achieve this by using the polynomial recursions 
arising from the sum sides of these identities.
Our methods of verification are discussed in detail in Section~\ref{sec:500}.

The idea of using computer explorations to discover and prove 
partition identities has been around for a while. 
In fact, methods similar to ours have been previously
used effectively by Andrews in \cite{A2}.
Later, Andrews also made creative use of SCRATCHPAD to this end. 
We refer the reader to Chapter 10 of \cite{A3}.
Computer searches have been used extensively in finding
finite Rogers-Ramanujan type identities (see~\cite{S} and~\cite{MSZ}).
In~\cite{MSZ}, the question of automatically discovering analytic identities was considered.
There is now a vast literature in which computer methods have been 
used in relation to partition identities.
See the introduction of~\cite{MSZ} for an extensive history of computer searches
related to Rogers-Ramanujan type identities.

Partition identities of this shape are fundamentally related to various
areas of mathematics and this is one of the reasons why these 
identities are extremely deep and important. We choose a few such areas to
make our point.

Lepowsky-Wilson's seminal work \cite{LW1}-\cite{LW4}
showed how these identities arise in what came to be understood as a 
vertex-operator-algebraic framework.
They showed how the Euler, Rogers-Ramanujan, Gordon, and Andrews-Bressoud
identities are related to the structure of standard modules for the affine
Lie algebra $A_1^{(1)}$, by means of vertex-operator-theoretic structures
that they called $Z$-algebras. In fact, they provided a  
vertex-operator-theoretic proof of the Euler and the Rogers-Ramanujan identities.
Carrying this program forward, Capparelli discovered new identities using certain 
standard modules for the affine Lie algebra $A_2^{(2)}$. 
See \cite{MP1}--\cite{MP3} for further ideas and discoveries
in this spirit.
Preliminary investigations reveal that
three out of our six new identities are related to the
principally specialized characters of level 3 
standard modules for the affine Lie algebra $D_4^{(3)}$. 
For the remaining three identities, the congruence conditions in the
product sides are ``asymmetric,'' i.e., they don't fall into
plus/minus pairs of congruence classes.
Therefore, these identities are not directly related to principally
specialized characters for affine Lie algebra modules.

Baxter's work~\cite{B} showed how the hard hexagon model from statistical mechanics
leads to the Rogers-Ramanujan identities. Significant advances in 
this area have been made by several authors, including 
Berkovich, Forrester, McCoy, Schilling, Warnaar, and others. 
See the references in~\cite{S}.
Recently, in \cite{GOW}, Griffin-Ono-Warnaar have found a framework incorporating
the Hall-Littlewood polynomials that extends 
the Rogers-Ramanujan identities to doubly infinite families of $q$-series identities
and illuminates certain arithmetic properties of such identities.

This paper is organized as follows:
Section~\ref{sec:prelim} contains mathematical preliminaries
necessary for our work.
Section~\ref{sec:methods} outlines our methods, including descriptions of
the main conditions used in \idf.
Section~\ref{sec:results} contains our main results: six new conjectures.
Section~\ref{sec:500} discusses further ways to computationally verify our conjectures.
Finally, Section~\ref{sec:future} provides ideas for extending our work.

\section{Preliminaries}
\label{sec:prelim}

A partition of $n$ is a list of integers
$\left(\lambda_1,\lambda_2,\dots,\lambda_m\right)$
such that $\lambda_1+\cdots+\lambda_m=n$
and $\lambda_1 \ge \lambda_2 \ge \cdots \ge \lambda_m \ge 1$.
In most of what follows, $n=0$ will be assumed to have
exactly one partition, namely, the null partition.

In this paper, a ``product side'' is a generating function of the form
$\prod_{j\ge1} \left(1-q^j\right)^{p_j}$.
Typically, each $p_j$ will be either 0 or $-1$ and this will correspond to
the generating function for partitions where the allowable parts are those 
$j$ such that $p_j = -1$. In addition, if the $p_j$s form a periodic sequence, 
then this product can be interpreted as the generating function for partitions 
whose parts satisfy certain congruence conditions modulo the period.

The sum side, on the other hand, deals with certain difference conditions
between  parts, along with some initial conditions and (possibly) other conditions.
For example, the sum side of Euler's ``odd-equals-distinct'' identity is that all parts must be distinct; this
means that $\lambda_i > \lambda_{i+1}$ for all $i$.

We recall Euler's algorithm, which is one of the crucial ingredients used in \idf:

\begin{prop} (Cf.\ Theorem 10.3 of \cite{A3}.)
\label{prop:euler}
Let $f(q)$ be a formal power series such that
\begin{equation}
f(q) = 1 + \sum\limits_{n\ge 1} b_nq^n.
\end{equation}
Then
\begin{equation}
f(q) = \prod\limits_{m\ge 1}(1-q^m)^{-a_m},
\end{equation}
where the $a_m$s are defined recursively by:
\begin{equation}\label{eqn:prod-recurrence}
nb_n = na_n + \sum\limits_{d|n,\, d<n}da_d + \sum\limits_{j=1}^{n-1}
\left(\sum\limits_{d|j} da_d\right)b_{n-j}.
\end{equation}
\end{prop}

\begin{proof}
See the proof of Theorem 10.3 of \cite{A3}. 
The basic idea is to use logarithmic differentiation.
\end{proof}

\begin{rmk}\label{rmk:garvan}
The algorithm for ``factoring'' a sum side into an infinite product indicated by equation 
\eqref{eqn:prod-recurrence} is a part of Frank Garvan's {\tt qseries} package.
\end{rmk}

\begin{cor}\label{cor:finiteisok}
Let $f(q)$ and $g(q)$ be formal power series with constant term 1
such that $$f(q)-g(q)\in q^{k+1}\mathbb{C}[[q]]$$ for some $k \ge 1$.
If 
\begin{equation}
f(q) = \prod\limits_{m\ge 1}(1-q^m)^{-a^{(f)}_m}
\end{equation}
and
\begin{equation}
g(q) = \prod\limits_{m\ge 1}(1-q^m)^{-a^{(g)}_m}
\end{equation}
then $a^{(f)}_m = a^{(g)}_m$ for all $m=1,\dots,k$.
\end{cor}
\begin{proof}
Equation \eqref{eqn:prod-recurrence} implies that $a_n$
is dependent only on $b_1,\dots,b_n$.
\end{proof}

\begin{rmk}\label{rmk:approx}
Computers can really only deal with polynomials as opposed to genuinely 
infinite formal power series. Corollary \ref{cor:finiteisok} implies that
giving a better polynomial approximation to a formal power series by
providing more terms with higher powers of $q$ does not affect the 
first few factors in the already computed product. 
This fact will be used in our verifications of the newly found identities
up to high order.
\end{rmk}

\section{Methods}
\label{sec:methods}

We use the following methodology to arrive at
interesting identities:

\begin{enumerate}
 \item 
First, we use the following conditions with varying parameters to generate
our list of sum sides:

\begin{itemize}
\item Smallest part size:
This is the smallest allowable part, along with the maximum 
multiplicity with which it is allowed to appear.
For example, in Gordon's identities, there is a restriction on the number of 
occurrences of 1 as a part.
In \idf, this is implemented as the procedure 
{\tt{SmPartCheck}}.

\item Difference-at-a-distance:
Fix values of $k$ and $d$. If, for all $j$, $\lambda_j - \lambda_{j+k}\geq  d$,
we say that the partition satisfies the difference $d$ at distance $k$ condition.
For example, the two Rogers-Ramanujan identities both feature
difference 2 at distance 1 conditions.
In \idf, this is implemented as the procedure 
{\tt{DiffDistCheck}}.

\item Congruence-at-a-distance:
If for all $j$,
$\lambda_j \le \lambda_{j+A} + B$
only if
$\lambda_{j} + \lambda_{j+1} \cdots + \lambda_{j+A}$ is congruent to $C \pmod D$,
we say that the partition satisfies the $\left(A,B,C,D\right)$-congruence condition.
In \idf, this is implemented as the procedure 
{\tt{CapparelliCheck}}. 
Using $A=1$ is an important specialization.
In this case, the condition simplifies to: two consecutive parts differ by at most
$B$ only if their sum is congruent to $C \pmod D$.
For example, in our notation, the key condition in Capparelli's identities
is the $\left(1,3,0,3\right)$-congruence condition.
The Andrews-Bressoud identities use a congruence condition with $D=2$.
Sometimes, it is important to consider seemingly ``wrong'' congruence conditions. 

Formal $q$-series corresponding to the ``wrong'' congruence conditions,
called ``ghost series'' in \cite{KLRS}, played a crucial role in the motivated proof of the 
Andrews-Bressoud identities in \cite{KLRS}.

\end{itemize}

Typically, we calculate up to $N$ terms, i.e., we examine
all partitions of $n$ from $0$ to $N$, and deduce
initial terms for the corresponding (infinite) generating function.
We arrive at:
$$1 + \sum\limits_{n=0}^{N}b_nq^n.$$
In our search, a usual value for $N$ is taken to be $30$.

\item Now, we use Euler's algorithm (see Proposition \ref{prop:euler}
and Remark \ref{rmk:garvan}) to factor the expression just obtained
as:
$$f(q) = \prod\limits_{m\ge 1}(1-q^m)^{-a_m}.$$
In light of Remark \ref{rmk:approx},
we concentrate on the sequence $\left\lbrace a_m\right\rbrace_{m=1}^{N}$.

\item The sequences $\left\lbrace a_m\right\rbrace_{m=1}^{N}$ that are periodic
with a sufficiently small period correspond to potential candidates for partition identities.

\item Given the potential candidates, we investigate further, using
either or both of the following strategies:

\begin{itemize}
 \item We increase the value of $N$ to incorporate more terms.
 \item We write down the recursions which govern the sum sides (see Section \ref{sec:500} 
for specific examples) and calculate hundreds of terms.
\end{itemize}

\end{enumerate}

\section{Results}
\label{sec:results}
Naturally, our methods rediscover many known identities.
So far, our searches have found six apparently new (conjectured) identities.
Three of them form a single family of mod 9 identities:
\begin{itemize}
\item[$I_1:$] The number of partitions of a non-negative integer into parts congruent to $1,$ $3,$ $6,$ or $8$ mod 9
is the same as the number of partitions with difference at least 3 at distance 2 such that
if two consecutive parts differ by at most 1, then their sum is divisible by 3.
\item[$I_2:$] The number of partitions of a non-negative integer into parts congruent to $2,$ $3,$ $6,$ or $7$ mod 9
is the same as the number of partitions with smallest part at least 2 and difference at least 3 at distance 2
such that if two consecutive parts differ by at most 1, then their sum is divisible by 3.
\item[$I_3:$] The number of partitions of a non-negative integer into parts congruent to $3,$ $4,$ $5,$ or $6$ mod 9
is the same as the number of partitions with smallest part at least 3 and difference at least 3 at distance 2
such that if two consecutive parts differ by at most 1, then their sum is divisible by 3.
\end{itemize}
Note that the congruence conditions are all symmetric. They could be rewritten as $\pm1$ and $\pm3$, $\pm2$ and $\pm3$, and
$\pm3$ and $\pm4$, respectively.
A fourth mod 9 identity appears to be connected, but has asymmetric congruence conditions:
\begin{itemize}
\item[$I_4:$] The number of partitions of a non-negative integer into parts congruent to $2,$ $3,$ $5,$ or $8$ mod 9 is the same as the number
of partitions with smallest part at least 2 and difference at least 3 at distance 2 such that if two consecutive parts differ by at most 1, then their sum is congruent to $2 \pmod 3$.
\end{itemize}

We also found a pair of mod 12 identities, again with asymmetric congruence conditions:
\begin{itemize}
\item[$I_5:$] The number of partitions of a non-negative integer into parts congruent to $1,$ $3,$ $4,$ $6,$ $7,$ $10,$ or $11$ mod 12
is the same as the number of partitions with at most one appearance of the part 1 and difference at least 3 at distance 3
such that if parts at distance two differ by at most 1, then their sum (together with the intermediate part) is congruent to $1 \pmod 3$.
\item[$I_6:$] The number of partitions of a non-negative integer into parts congruent to $2,$ $3,$ $5,$ $6,$ $7,$ $8,$ or $11$ mod 12
is the same as the number of partitions with smallest part at least 2, at most one appearance of the part 2, and
difference at least 3 at distance 3 such that if parts at distance two differ by at most 1,
then their sum (together with the intermediate part) is congruent to $2 \pmod 3$.
\end{itemize}

All six of these conjectures have been verified for at least 500 terms.

\section{Verification}
\label{sec:500}

In order to verify our conjectures up to a higher degree of certainty, 
we make use of the recursions which
govern the respective sum sides, similarly to \cite{A4}. 
Specifically, we use polynomial generating functions which
incorporate the conditions on the sum sides, with added
restrictions on the largest part which can appear in the partition.
In the limit as the largest part goes to infinity, these 
polynomials converge to their respective (formal) sum sides.
Actually, polynomials in this spirit have also been used in the ``motivated proofs''
of various identities. In \cite{AB}, such polynomials, denoted by $A$, 
were used in the motivated proof of the Rogers-Ramanujan identities.
In \cite{LZ}, \cite{CKLMQRS}, and \cite{KLRS}, analogous
polynomials, denoted by $h$, were used in the motivated proofs of the
Gordon, G\"ollnitz-Gordon-Andrews, and Andrews-Bressoud identities, respectively.

Let $f_j(q)$ be the generating function for the partitions counted in the
sum side of the identity $I_j$ for $j=1,2,3$.
Let 
\begin{align*}
P_{j,k}(q) = & \text{ generating function for the partitions counted }\\
& \text{ in the sum side of the identity } I_j, \text{ where }  j=1,2,3, \\
&  \text{ with the added restriction that the 
largest part is at most } k.
\end{align*}

It is clear that $P_{j,k}(q)$ is a polynomial, which agrees with $f_j(q)$ 
up to  $q^k$. Therefore, by Corollary \ref{cor:finiteisok},
the product obtained by applying Euler's algorithm to $P_{j,k}(q)$
agrees with the product corresponding to $f_j(q)$ up to the factor $(1-q^k)^{-a_k}$.
It is easy to see that the polynomials $P_{j,k}(q)$ satisfy the same recursion
irrespective of the value of $j$:
\begin{align}
P_{j,3n} &= P_{j,3n-1} + q^{3n}P_{j,3n-2}+ q^{3n}  q^{3n}  P_{j,3n-3} \label{eq:Crec0} \\
P_{j,3n+1} &= P_{j,3n} + q^{3n+1}P_{j,3n-1}\label{eq:Crec1}\\
P_{j,3n+2} &= P_{j,3n+1} + 
q^{3n+2}  q^{3n+1}  P_{j,3n-1}+ 
q^{3n+2}  q^{3n}   P_{j,3n-2}+ 
q^{3n+2}P_{j,3n-1}. \label{eq:Crec2}
\end{align}
Note that we have presented some exponents in an unsimplified form in our recursions (above and below)
to better illustrate how the latter are produced.
The initial conditions, however, depend on the value of $j$:
\begin{alignat}{3}
P_{1,1} &=  1+q \quad &  \quad  P_{1,2} &= 1+q+q^2+q^3  \quad  & \quad  P_{1,3}  &= 1+q+q^2+2q^3 + q^4 + q^7\\
P_{2,1} &=  1 \quad &  \quad  P_{2,2} &= 1+q^2  \quad  & \quad  P_{2,3}  &= 1 + q^2+q^3  + q^6\\
P_{3,1} &=  1 \quad &  \quad  P_{3,2} &= 1  \quad  & \quad  P_{3,3}  &= 1 + q^3+q^6. 
\end{alignat}

For the remaining mod 9 identity, $I_4$, 
letting
\begin{align*}
Q_{k}(q) = & \text{ generating function for the partitions counted }\\
& \text{ in the sum side of the identity } I_4, \\
&  \text{ with the added restriction that the 
largest part is at most } k,
\end{align*}
the recursions are: 
\begin{align}
Q_{3n} &= Q_{3n-1} + 
q^{3n}  q^{3n-1}  Q_{3n-3} + 
q^{3n}  q^{3n-2}  Q_{3n-4} + 
q^{3n}Q_{3n-3} \\
Q_{3n+1} &= Q_{3n} + 
q^{3n+1}  q^{3n+1}  Q_{3n-2} + 
q^{3n+1} Q_{3n-1} \\
Q_{3n+2} &= Q_{3n+1} + q^{3n+2}Q_{3n},
\end{align}
with the initial conditions:
$$Q_0 =1, \,\, Q_1 = 1, \,\, Q_2=1+q^2, \,\, Q_3=1+q^2+q^3+q^5.$$

Now, let us  turn to the mod 12 identities.
For the first mod 12 identity, $I_5$, let
\begin{align}
R_{n,a}(q) = &\text{ generating function of partitions with largest part at most } n
\nonumber\\
&\text{ and at most } a \text{ parts equalling } n \nonumber\\
&\text{ in addition to the given constraints in the sum-sides.}
\end{align}
It is easy to see that we need only consider $a\in\{1,2\}$.
The following recursions and initial conditions are satisfied by these polynomials:

\begin{align}
R_{n,1} &= 
q^n  R_{n-1,2} - q^n  q^{n-1}  q^{n-2}  q^{n-2}  R_{n-4,1} + R_{n-1,2}\\
R_{n,2} &= 
q^n   q^n  R_{n-2,1} + R_{n,1}\\
R_{1,1}&=  1+q\\
R_{1,2}&=  1+q\\
R_{2,1}&=  1+q+q^2+q^3\\
R_{2, 2}&= 1+q+q^2+q^3+q^4\\
R_{3, 1}&= 1+q+q^2+2q^3+2q^4+q^5+q^6+q^7\\
R_{3, 2}&= 1+q+q^2+2q^3+2q^4+q^5+2q^6+2q^7\\
R_{4,1}&=  1+q+q^2+2q^3+3q^4+2q^5+3q^6+4q^7+2q^8+q^9+2q^{10}+q^{11}\\
R_{4, 2}&= 1+q+q^2+2q^3+3q^4+2q^5+3q^6+4q^7+3q^8+2q^9+3q^{10}+2q^{11}.
\end{align}

With a similar definition of
the polynomials $S_{n,a}(q)$, for the second mod 12 identity, 
we have the following recursions and initial
conditions:
\begin{align}
S_{n,1}&= q^n S_{n-1,1} + S_{n-1,2}\\
S_{n,2} &= q^n  q^n  q^{n-1}  S_{n-3,2} + q^n   q^n   S_{n-2,1} + S_{n,1}\\
S_{1,1}&=1\\
S_{1,2}&=1\\
S_{2, 1}&= 1+q^2\\
S_{2, 2}&= 1+q^2\\
S_{3, 1}&= 1+q^2+q^3+q^5\\
S_{3, 2}&= 1+q^2+q^3+q^5+q^6+q^8.
\end{align}

For faster computation, we use the following trick:
Say we want to check that the partition identities hold
till $n=N$. Then, it suffices to 
compute the polynomials modulo $q^{N+1}$. 
This greatly hastens the computations.

\section{Further work}
\label{sec:future}
First, it is always possible to expand the parameter-space search
by incorporating more innovative conditions on the sum sides 
(for instance, the G\"ollnitz-Gordon-Andrews identities \cite{A1},
the Andrews-Santos identities \cite{AS}, etc.)
or by increasing the values of the various parameters. 
We hope that more identities could be found in this way.
It would be interesting to examine the sum sides of more 
recent partition identities to see if they can inspire additional 
checks to build in to the package. 
For one intriguing example, see the new difference conditions 
on the sum sides of the conjectures in Nandi~\cite{N}.

Recent research (see, for instance, \cite{CM}) has focused on proving
overpartition analogues of many classical partition identities.
It would be interesting to extend our methods to consider overpartitions
and more generally, multi-color partition identities.
However, one main challenge is that there are many more overpartitions
than partitions of a given integer. Using a na\"ive approach,
it may be difficult to calculate out enough terms to 
form reasonable conjectures.

Finally, proving the six conjectures mentioned in 
the previous section is a task in progress.

\section{Acknowledgments}
The authors gratefully acknowledge suggestions from Doron Zeilberger,
who suggested an improvement to our program which allowed for verification of many more terms
of our conjectures. 
We used Frank Garvan's implementation of Euler's algorithm 
in his Maple package {\tt {qseries}}~\cite{G}.
We are indebted to James Lepowsky for his constant encouragement,
suggestions for improving the exposition,
and for introducing the first author to the delightful world of partitions
and vertex operator algebras. We thank Drew Sills for proposing 
ways for proving these identities, for suggesting relevant references,
and for carefully reading the manuscript.
We also thank Robert McRae for illuminating discussions.

\end{document}